\documentclass[a4paper,reqno,12pt]{amsart}
\usepackage{amsmath}
\usepackage{amsfonts}
\usepackage{amssymb,mathrsfs,physics}
\usepackage{graphicx,hyperref}
\usepackage{enumerate,xcolor}
\hypersetup{colorlinks=true,linkcolor=blue, anchorcolor=green, citecolor=green, urlcolor=red, filecolor=magenta, pdftoolbar=true}
\usepackage[margin=1in]{geometry}
\DeclareGraphicsExtensions{.eps}
\newtheorem{theorem}{Theorem}[section]
\newtheorem{lemma}[theorem]{Lemma}
\newtheorem{proposition}[theorem]{Proposition}
\newtheorem{corollary}[theorem]{Corollary}
\theoremstyle{definition}

\numberwithin{equation}{section}
\newcommand{\sS}{\sigma_S}
\newcommand{\tl}[1]{\widetilde{#1}}

\newcommand{\ran}{\operatorname{ran}}

\title[Quaternionic Quantum $q$-oscillator]{Quaternionic Quantum $q$-Oscillator and Unbounded Subnormal Operators}
\author{T. Prasad}
\address{Department of Mathematics,  Nalanda University, Rajgir-803116, Bihar, India}
\email{prasadvalapil@gmail.com, prasad@nalandauniv.edu.in}
\author{K Krishnan}
\address{Department of Mathematics, University College (Affiliated to University of Kerala), Thiruvananthapuram, Kerala, India}
\email{krishnank@universitycollege.ac.in}
\author{E. Shine Lal}
\address{Department of Mathematics, University College (Affiliated to University of Kerala), Thiruvananthapuram, Kerala, India}
\email{shinelal@universitycollege.ac.in}
\begin{document}
	\maketitle
	\begin{abstract}
		In this note, we observe that the class of  unbounded subnormal operators on quaternionic Hilbert spaces is a  solution of quaternionic q-oscillator problem whose complex case is addressed by Szefraniec in \cite{q-osci}. We also address the connection of unbounded subnormal operators on quaternionic Hilbert spaces, quaternionic quantum harmonic q-oscillator  and   quaternionic quantum  Economics  using model theory of unbounded subnormal operators. Moreover, we study model theory for sub-2-normal operators on quaternionic Hilbert space.
	\end{abstract}
	
	\section{Introduction and Preliminaries}
	Complex Quantum Economics (CQE) inculcates Quantum Physics models in quantum Economics and treats asset prices as wave functions existing in a state of inertia until a transaction perturbs the system ~\cite{orrell2022,orrell}. CQE can consider one asset attribute at a time. This limitation forces multi-asset dependencies and cross-correlations to be modeled using separate, exogenous covariance matrices. In this paper, we consider quaternionic quantum harmonic oscillators that can be introduced in quantum Economics, which can treat multi asset factors simultaneously.   The algebraic non-commutativity gives an ideal mathematical model for market characterestics, where the order of trades or assets affects the output of the market value. 
	
	Birkhoff and von Neumann \cite{von N} introduced quaternionic skew fields $\mathbb{H}$ and has been extensively developed by Adler~\cite{adler}. The quaternionic division algebra is the span of  three distinct anti-commuting imaginary units $\{i, j, k\}$ satisfying,
	\[
	i^2 = j^2 = k^2 = -1, \quad
	ij = -ji = k, \quad jk = -kj = i, \quad ki = -ik = j.
	\]
	The unit sphere of imaginary quaternions is defined as:
	\[
	\mathbb{S} = \{ q \in \mathbb{H} \ : \ q^2 = -1 \}.
	\]
	For any $J \in \mathbb{S}$, we define the slice $\mathbb{C}_J:=\{q\in \mathbb{H}: q=x+Jy, x,y\in \mathbb{R}\}$ and is structurally isomorphic to the complex field $\mathbb{C}$. It is known that any quaternion $q \in \mathbb{H}$ can be uniquely written as $q = x + Jy$, for some $J\in \mathbb{S}$, where $x, y \in \mathbb{R}$ and $y \geq 0$ \cite{adler}.
	
	Let $\Omega \subseteq \mathbb{H}$ be open. A function $f: \Omega \to \mathbb{H}$ is called {left slice regular} if for every imaginary unit $J \in \mathbb{S}$, its restriction $f_J$ to the complex plane $\mathbb{C}_J$ satisfies:
	\[
	\frac{\partial}{\partial x} f_J(x + Jy) + J \frac{\partial}{\partial y} f_J(x + Iy) = 0.
	\]
	Equivalently, $f$ can be represented via its power series expansion:
	\[f(q) = \sum_{n=0}^{\infty} q^n a_n, \quad a_n \in \mathbb{H}.\]
	
	Since the quaternionic multiplication is non-commutative, the standard directional derivative does not preserve slice regularity. Hence, we define {slice derivative} ( $\partial_S$ derivative) of a slice regular function $f$ as,
	\[	\partial_S f(q) = \frac{\partial}{\partial x} f(x + Jy).\]
	See ~\cite{colombo2016, colombo2009, sabadini} for more details.
	It can be seen that if $f$ is left slice regular, then its slice derivative $\partial_S f$ is also left slice regular.
	
	Let $\mathcal{C(H)}$ denote the space of all densely defined closed  right-linear operators on a right quaternionic Hilbert space $\mathcal{H}$. For an operator $T \in \mathcal{C(H)}$ and a quaternionic scalar $q = x + Jy$, the classical eigenvalue equation $T v = v q$ does not holds. Instead, we use the {$S$-spectrum}~\cite{colombo2009} which is defined as follows:
	
	For $T \in \mathcal{B}(V)$, we define the quaternionic operator $Q_q(T)$ associated with $q \in \mathbb{H}$ as:
	\[	Q_q(T) = T^2 - 2xT + |q|^2 {I},\]
	where ${I}$ is the identity operator and $x = \text{Re}(q)$. The {$S$-resolvent set} $\rho_S(T)$ is defined as:
	\[	\rho_S(T) = \{ q \in \mathbb{H} \ : \ Q_q(T) \text{ is invertible in } \mathcal{B(H)} \}.\]
	The {$S$-spectrum} $\sigma_S(T)$ is the complement of the $S$-resolvent set in $\mathbb{H}$:
	\[	\sigma_S(T) = \mathbb{H} \setminus \rho_S(T).\]
	The $S$-spectrum satisfies the {axially symmetric} property, that is, if a point $q = x + Jy$ belongs to $\sigma_S(T)$, then the entire $[q]$ is properly contained in $\sigma_S(T)$ where $[q]$ is the $2$-sphere defined by $x + \mathbb{S}y$.
	
	Let $S$ be a densely defined, right linear operator on $\mathcal{H}$. The {quaternionic adjoint} $S^*$ is defined on the domain 
	\[
	\mathcal{D}(S^*) = \{ g \in \mathcal{H} : \exists \, g^* \in \mathcal{H} \text{ such that } \langle Sf, g \rangle = \langle f, g^* \rangle \,\quad \forall f \in \mathcal{D}(S) \}.
	\]
	For such $g$, we uniquely set $S^*g = g^*$ \cite{colombo2016}.	A densely defined right linear operator $N$ on a right quaternionic Hilbert space $\mathcal{K}$ is {normal} if $N$ is closed and satisfies $N^*N = NN^*$. It guarantees that $\mathcal{D}(N) = \mathcal{D}(N^*)$ and $\|Nf\| = \|N^*f\|$ for all $f \in \mathcal{D}(N)$ \cite{colombo2009}.
	
	The class of subnormal operators on Hilbert space was introduced by Halmos\cite{halmos1950} to fill the inevitable gap between classes of normal operators and hyponormal operators. Many authors including Bram, Conway, et. al. extensively studied subnormal operators on complex Hilbert spaces \cite{bram1955,subnormal,halmos1950}. Curto and Prasad \cite{curto} introduced  sub-$n$-normal and $n$-subnormal classes of operators which are the extensions of subnormal operators and developed a model theory for sub-$2$-normal operators in \cite{curtoprasad}. A right linear quaternionic operator $S$ defined on a quaternionic Hilbert space $\mathcal{H} $ is said to be sub-$2$-normal if it has an extension $N$ such that $N^2$ is normal.
	
	An operator $S\in \mathcal{C(H)}$ is said to be subnormal if  there exist a normal operator  $N\in \mathcal{C(K)}$ where $\mathcal{H} \subseteq \mathcal{K}$ such that $N|_{D(S)}x=Sx$.  We invoke the concept of tightness in quaternionic settings by a similar manner in complex Hilbert space settings considered by Szefraniec in \cite{q-osci}. Let $S$ be a right linear subnormal operator on $\mathcal{H}$ and let $N$ acting on $\mathcal{K} \supseteq \mathcal{H}$ be a normal extension of $S$. The extension $N$ is said to be  tight if $\mathcal{D}(N) \cap \mathcal{H}= \mathcal{D}(\overline{S})$ and $\ast$-tight if $\mathcal{D}(N^*) \cap \mathcal{H} = \mathcal{D}(S^*)$
	where $\overline{S}$ denotes the closure of $S$.
	
	A densely defined right linear subnormal operator $S$ on $\mathcal{H}$ has a normal extension which is both tight and $\ast$-tight if and only if
	\begin{equation*}
		\mathcal{D}(\overline{S}) = \mathcal{D}(S^*).
	\end{equation*}
	Let $(\Omega,\mu)$ be a measure space where $\Omega \subset \mathbb{H}$ be an axially symmetric quaternionic domain.
	\[L^2(\Omega,\mu,\mathbb{H})=\{f:\Omega \to \mathbb{H},\int\limits_\Omega |f(q)|^2 d\mu(q) < \infty\}\]  denotes the collection of all square integrable quaternionic valued functions on $\Omega$. See \cite{ST for normal}. $L^2(\Omega,\mu,\mathbb{H})$ defines a right quaternionic Hilbert space equipped with the inner product
	\[\langle f,g\rangle= \int\limits_\Omega \overline{f(q)}g(q)d\mu(q)\]
	and $H^2(\Omega,\mu,\mathbb{H})$ denotes the quaternionic Hardy-Hilbert space defined by
	\[H^2(\Omega,\mu,\mathbb{H})=\left\{\sum\limits_{n=0}^\infty q^n a_n: \|\sum\limits_{n=0}^\infty q^n a_n\|_{H^2(\Omega,\mu,\mathbb{H})} = \sum\limits_{n=0}^\infty |a_n|^2 < \infty \right\}.\] 
	
	Extending this notion, a vector valued Hardy space consists of functions of the form $\Phi=\begin{bmatrix}
		\varphi_1 &\varphi_2
	\end{bmatrix}$ where $\Phi: \Omega \oplus \Omega \to \mathbb{H}^2$ such that $\varphi_1, \varphi_2 \in H^2(\Omega,\mu,\mathbb{H})$. 
	
	Let $\phi \colon \Omega \to \mathbb{C}$ be a measurable function. We define the operator $M_\phi \colon \mathcal{D}(M_\phi) \subseteq L^2(\Omega,\mu, \mathbb{H}) \to L^2(\Omega,\mu, \mathbb{H})$ by
	\[
	M_\phi(f)(q) = \phi(q) \cdot f(q), \quad \text{for all } f \in \mathcal{D}(M_\phi),
	\]
	where
	\[
	\mathcal{D}(M_\phi) = \left\{ f \in L^2(\Omega,\mu, \mathbb{H}) \colon \phi \cdot f \in L^2(\Omega,\mu, \mathbb{H}) \right\}.
	\]
	
	For more details, see \cite{ST for normal,Viswanath}.	Note that the polynomial multiplication operator $M_p(x)$ can be considered as $M_p(x)=\sum q^n(x)\cdot a_n.$
	
Let $J$ be imaginary ($J^2=-I$)  and $i\in\mathbb{S}$. Then $\mathcal{H}_+^J=\{x:Jx=xi\}$ is a closed $\mathbb{C}_i$-subspace on which $\ip\cdot_i$ is $\mathbb{C}_i$-valued, hence a complex Hilbert space. For $j\in\mathbb{S}$ with $ij=-ji$ one has $\mathcal{H}=\mathcal{H}_+^J\oplus\mathcal{H}_+^Jj$ (orthogonal) and $\mathcal{H}_+^Jj=\{x:Jx=-xi\}$ \cite{GMP13, ST for normal}. Also, $T\mapsto T_+:=T|_{\mathcal{H}_+^J}$ is a bijective, norm-preserving, real-linear unital $*$-homomorphism from $\{T\in\mathcal{B}(\mathcal{H}):TJ=JT\}$ onto $\mathcal B(\mathcal{H}_+^J)$.
	Every complex Hilbert space $\mathcal X$ is the slice space of some quaternionic Hilbert space $\mathcal X_\mathbb{H}$ \cite{ST for normal}. 
	Note that every normal operator $N\in\mathcal{B}(\mathcal{H})$ commutes with some imaginary operator $J$. A normal operator $N$ has a unique spectral measure $E_N$ on the axially symmetric Borel subsets of $\sS(N)$, with values in orthogonal projections of $\mathcal{B}(\mathcal{H})$, and $N=\int\lambda\,dE_N$. For an  axially symmetric set $\Omega$, $E_N(\Omega)=E^i(\Omega\cap\mathbb{C}_i)$ where $E^i$ is the complex spectral measure of the complex normal operator $N_i$. If $S$ commutes with $N$, then $S$ commutes with all $E_N(\Omega)$ \cite{Viswanath}.  For a self-adjoint operator $P$, $\sS(P)\subset\mathbb{R}$, every Borel subset of $\mathbb{R}$ is axially symmetric, and bounded Borel functions $f(P)\in\mathcal{B}(\mathcal{H})$ are defined. See \cite{CGK18, GMP13, ST for normal} for more details.
	
Let $\Psi$ be a quaternionic wave function. Then \[\Psi=\Psi_0+J \Psi_1\] where $J\in \mathbb{S}$ and $\Psi_0$ and $\Psi_1$ are left slice regular functions. The quaternionic Hamiltonian operator, $\hat{H}$ is given by the equation,
	\[\hat{H}=\dfrac{\hbar^2}{2m}\dfrac{\partial^2}{\partial x^2}+\dfrac{1}{2}K\hat{x}^2+UI\]
	where $K$ and $U$ are quaternionic constants \cite{giardino2026}. The corresponding annihilation and creation operators are respectively given by,
	\begin{align*}
		\hat{a}&=\sum\limits_{z=i,j,k} \sqrt{\frac{m\omega}{2\hbar}}\left(\hat{x}_z+\frac{z}{m\omega}\hat{p}_z\right)\\
		\hat{a}^\dagger&=\sum\limits_{z=i,j,k} \sqrt{\frac{m\omega}{2\hbar}}\left(\hat{x}_z-\frac{z}{m\omega}\hat{p}_z\right)
	\end{align*}
	where $\hat{x}=\sum\limits_{z=i,j,k}z\hat{x}_z $ and $\hat{p}=\sum\limits_{z=i,j,k}z\hat{p}_z $. Also we have,
	\[[\hat{a},\hat{a}^\dagger]=\hat{a}\hat{a}^\dagger-\hat{a}^\dagger\hat{a}=1+V,\]
	where $V$ is a quaternionic term called cross terms \cite{giardino2026}. Hence the usual canonical commutation relation doesnot holds. therefore, it is necessary that we have to restrict the quaternionic equation to slices. 
	
	Assume that we restrict to the complex plane $\mathbb{C}_J$. Then $\hat{x}_J$ is the position operator and corresponding momentum operator is given by,
	\[\hat{p}_J=-J \hbar \partial_S\] where $\partial_S$ denotes the slice derivative. So the respective slice Hamiltonian operator is given by,
	\[\dfrac{\hat{p}_J^2}{2m}+\frac{1}{2}m \omega^2 \hat{x}_J^2.\]
	The corresponding annihilation operator $\hat{a}_J$ and creation operator $\hat{a}_J^\dagger$ satisfies the canonical commutation relation.
	
	Let $\mathcal{H}_1$ and $\mathcal{H}_2$  be quaternionic Hilbert spaces, and let $A:\mathcal{D}(A)\subseteq\mathcal H_1\to\mathcal H_2$ be a densely defined, closed linear operator. The associated Dirac-type block operator $D$ on the Hilbert space $\mathcal H=\mathcal H_1\oplus\mathcal H_2$ is defined by
	\[
	D=\begin{pmatrix}
		0&A^*\\
		A&0
	\end{pmatrix}.
	\]
	with domain is $\mathcal{D}(D)
	=\mathcal{D}(A)\oplus\mathcal{D}(A^*)$ \cite{dirac1}. Dirac operator is a first-order differential operator whose square gives a second-order operator similar to the Hamiltonian of a quantum harmonic oscillator. In supersymmetric quantum mechanics, it can be written in block form as $D=\begin{pmatrix}0&A^*\\A&0\end{pmatrix}$, where $A$ and $A^*$ act like anhilation and creation operators. For the harmonic oscillator, choosing $A$ and $A^*$ suitably, this connects the Dirac-type block operator to the  Hamiltonian oscillator and its supersymmetric structure.
	
	In this paper, we considered unbounded quaternionic subnormal operators and give model theorem for this class. Also, quaternionic quantum $q$-oscillators are discussed and we prove that unbounded quaternionic subnormal operators belongs to the solution space of quantum $q$-oscillator. We inculcate quaternionic quantum $q$-oscillator into quantum Economics and observe that  quaternionic Hamiltonian operators can be identified with  multiplication operators on quaternionic Hardy spaces. 
	
	\section{Extensions of Unbounded Subnormal Operators on Quaternionic Hilbert Space}
	In \cite{sqho}, authors proved that bounded quaternionic cyclic subnormal operators are unitarily equivalent to multiplication operators on quaternionic Hardy space. In this section, we consider a model theory for unbounded quaternionic subnormal operators.  Later, we deals with the class of $2$-normal operators, which contains the class of normal operators. A right linear quaternionic operator $S$ defined on a quaternionic Hilbert space $\mathcal{H} $ is said to be sub-$2$-normal if  there exists a quaternionic Hilbert space $\mathcal{K}\supseteq \mathcal{H}$ and a $2$-normal operator $N$ defined on $\mathcal{K}$ such that $N|_{\mathcal{H}}=S$ and N$\mathcal{H}\subseteq\mathcal{H}$, or simply $S$ has a $2$-normal extension. We studied the model theory for quaternionic sub-2-normal operators motivated by the study of sub-2-normal operators in complex Hilbert space by Curto and Prasad \cite{curtoprasad}, where the authors proved that complex sub-2-normal operators are unitarily equivalent to multiplication operators on vector valued Hardy spaces.
	
	A right linear quaternionic  operator $T\in \mathcal{B(H)}$ is said to be $2$-normal if $T^2$ is a quaternionic normal operator. 	The following result by Ramesh and Kumar in \cite{ST for normal} is used in the sequel.
	\begin{proposition}\cite{ST for normal}
		Let $N \in \mathcal{C(H)}$ be normal. Then there exist
		\begin{enumerate}
			\item a Hilbert basis $B$ of $\mathcal{H}$,
			\item a measure space $(\Omega, \nu)$,
			\item a unitary operator $V : \mathbb{H}\to L^2(\Omega,\mathbb{H},\nu)$, and
			\item a $\nu$-measurable function $\eta : \Omega \to \mathbb{C}$
		\end{enumerate}
		so that if $N$ is expressed with respect to $B$, then
		$Nx = U^*M_\eta Ux$, for all $x \in \mathcal{D}(N)$, where $M_\eta$ is quaternionic multiplication operator on $L^2(\Omega, \mathbb{H}\nu)$ induced by $\eta$, with the domain
		\[\mathcal{D}(M_\eta)= \{g \in L^2(\Omega, \mathbb{H},\nu): \eta \cdot g \in L^2(\Omega, \mathbb{H},\nu) \}.\]	
	\end{proposition}
	Using the above proposition and proof techniques in \cite{qho}, we give a model theorem for unbounded quaternionic subnormal operators.
	\begin{theorem}\label{M_z}
		Let $S \in \mathcal{C}(\mathcal{H})$ be a densely defined cyclic subnormal operator on a right quaternionic Hilbert space $\mathcal{H}$. Then there exists a Borel measure $\mu$ on the quaternionic domain $\mathbb{H}$ such that 
		\[S\psi = V^*M_\eta V \psi \qquad \text{ for every } \psi \in \mathcal{D}(S)\subseteq \mathcal{H}\] 
		where $M_\eta$ is the multiplication operator  acting on $H^2(\Omega,\mu, \mathbb{H})$, where ${H}^2(\Omega,\mu, \mathbb{H})$ denotes the closed right linear span of polynomials on $L^2(\Omega,\mu, \mathbb{H})$	with the domain
		\[\mathcal{D}(M_\eta)= \{f(\eta) \in H^2(\Omega, \mathbb{H},\nu): \eta \cdot f(\eta) \in H^2(\Omega, \mathbb{H},\nu) \}.\]
	\end{theorem}
	
	\begin{proof}
		Let $x \in \mathcal{H}$ be a cyclic vector for $S$, and let $N \in \mathcal{C}(\mathcal{K})$ be the minimal normal extension of $S$. Hence,
		\[	\overline{span}  \{ S^n x : n \geq 0 \} = \mathcal{D}(S) \subseteq \mathcal{H} \quad \text{ and } \overline{span} \{ N^{*j} N^k x : j, k \geq 0 \}=\mathcal{D}(N)\subseteq \mathcal{K}.\]
		By above Proposition, there exists a measure space $(\Omega, \nu)$,
		a unitary operator $V : \mathbb{H}\to L^2(\Omega,\mathcal{H},\nu)$, and  a $\nu$-measurable function $\eta : \Omega \to \mathbb{C}$ such that 
		\[Nx = U^*M_\eta Ux, \qquad \text{ for all } x \in \mathcal{D}(N).\]
		
		\noindent We denote the space of all quaternionic polynomials by $\mathbb{H}[\eta]$. Let  $p\in \mathbb{H}[\eta]$, where $p=p(\eta)= \sum \eta^n a_n$ where $a_n \in \mathbb{H}$.  We denote $M_p(x)= \sum \eta^n\cdot x\  a_n$. Now, for cyclic vector $x \in \mathcal{D}(S)$, we have
		\begin{align*}
			\|p(S)x\|^2_{\mathcal{H}} &= \langle p(S)x, p(S)x \rangle_{\mathcal{H}}  \\
			&= \langle p(N)x, p(N)x \rangle_{\mathcal{K}} \\
			&= \langle \sum M_{\nu^k} U(x) a_k, \sum M_{\nu^k} U(x) a_k \rangle_{L^2(\nu, \mathbb{H})}  \\
			&= \langle M_p(Ux), M_p(Ux) \rangle_{L^2(\Omega,\nu, \mathbb{H})} \nonumber \\
			&= \int_{\mathbb{H}} |p(\eta)|^2 |(Ux)(\eta)|^2 \, d\nu(\eta).
		\end{align*}
		If we let, $d\mu(\eta) = |(Ux)(\eta)|^2 \, d\nu(\eta)$, $\|p(S)x\|^2_{\mathcal{H}} = \|p\|^2_{L^2(\omega,\mu, \mathbb{H})}$.
		
		Let $\mathcal{P}_x=\{p(S)(x): p \in \mathbb{H}[z]\}$ and if we define $V_0: \mathcal{P}_x \to L^2(\Omega, \mu, \mathbb{H})$  by,
		\[V_0(p(S)x) = p(\eta),\]
		$V_0$ is a well-defined, isometric right linear operator on $\mathcal{P}_x$. Since $\mathcal{P}_x$ is dense in $\mathcal{D}(S)$, $V_0$ can be extended to  $V: \mathcal{D}(S) \to L^2(\Omega,\mu, \mathbb{H})$. Clearly $V$ is a continuous unitary extension of $V_0$, the range of $V$ is the closure of the polynomials under the $L^2(\Omega, \mu, \mathbb{H})$ norm, which is  ${H}^2(\Omega, \mu, \mathbb{H})$. Now,
		\[
		V S (p(S)x) = V (S p(S) x) = \eta \cdot p(\eta) \text{ and } M_\eta V (p(S)x) = M_\eta p(\eta) = \eta \cdot p(\eta).
		\]
		Hence $VS\psi = M_\eta V \psi$ holds for every $\psi \in \mathcal{D}(S)$. Hence, $S\psi = V^*M_\eta V \psi$ for every $\psi \in \mathcal{D}(S)\subseteq \mathcal{H}$.
	\end{proof}
	 Motivated by Radjavi and Rosenthal \cite{RR71}, we study the structure theorem for quaternionic 2-normal operators. We begin with the following lemmas.
	
	{\begin{lemma}\label{lem:sphspec}
			If $TJ=JT$ with $J$ imaginary, then $\sigma(T_i)=\sigma(T_+)\cup\overline{\sigma(T_+)}$ and $\sS(T)=[\sigma(T_+)]$. Hence $\sS(\tl S)=[\sigma(S)]$ for $S\in\mathcal B(\mathcal{H}_+^J)$.
		\end{lemma}
		\begin{proof}
			Let $\mathcal{H}_i=\mathcal{H}_+\oplus\mathcal{H}_+j$, and both summands are $T_i$-invariant. The map $x\mapsto xj$ is an antilinear bijection $\mathcal{H}_+\to\mathcal{H}_+j$ conjugating $T_+$ to $T|_{\mathcal{H}_+j}$. Hence, $\sigma(T|_{\mathcal{H}_+j})=\overline{\sigma(T_+)}$. Therefore, $\sS(T)\cap\mathbb{C}_i=\sigma(T_i)$, and $\sS(T)$ is axially symmetric.
	\end{proof}}
	\begin{lemma}\label{lem:nil}
		If $T\in\mathcal{B}(\mathcal{H})$ and $T^2=0$, there is an imaginary $J$ commuting with $T$.
	\end{lemma}
	\begin{proof}
		Let $L=(\ker T)^\perp$, $K_1=\overline{\ran T}\subset\ker T$, $K_0=\ker T\oplus K_1$. Then $T|_L=VC$ 
		with $C=|T|_L|\ge0$ is one-to-one on $L$ and $V$ is unitary from $L$ onto $K_1$. Then there is an imaginary $J_L$ on $L$ commuting with $C$, and an imaginary $J_0$ on $K_0$. Let $J=J_0\oplus VJ_LV^*\oplus J_L$ on $K_0\oplus K_1\oplus L$. For $x\in L$, $JTx=VJ_LCx=VCJ_Lx=TJx$  on $K_0\oplus K_1$.
	\end{proof}
	{\begin{proposition}\label{prop:J}
			Let $T\in\mathcal{B}(\mathcal{H})$ with $T^2$ normal. Then there is an imaginary operator $J$ with $TJ=JT$ (and so $T^*J=JT^*$).
	\end{proposition}}
	\begin{proof}
	Let $N=T^2$ and $N_2=\frac12(N-N^*)$. By Corollary~\ref{cor:fuglede}, $T$ and $T^*$ commute with $N,N^*,N_2,|N_2|$, and $\mathcal{H}=\mathcal{H}_{1}\oplus\mathcal{H}_2$ with $\mathcal{H}_2=\ker N_2$, $\mathcal{H}_{1}=\mathcal{H}_2^\perp$, both reducing $T$.	Let $N_2=V|N_2|$ be the polar decomposition. On $\mathcal{H}_{1}$, $\ker N_2=0$ and $\overline{\ran N_2}=\mathcal{H}_{1}$, hence $V$ is unitary. $N_2^*=-N_2$ and uniqueness of the polar decomposition give $V^*=-V$, and hence $J_{1}:=V|_{\mathcal{H}_{1}}$ is imaginary. $T$ commutes with $N_2$ and $|N_2|$, hence $TV|N_2|x=VT|N_2|x$ on the dense range of $|N_2|$, so $TJ_{1}=J_{1}T$. Now, $N_2:=N|_{\mathcal{H}_2}$ is self-adjoint and $T_2:=T|_{\mathcal{H}_2}$ reduces with $T_2^*$ commuting with $N_2$. The spectral projections of $N_2$ commute with $T_2,T_2^*$. Let $\mathcal{H}_0=\ker N_2$ and  let $\displaystyle \mathcal{H}_n^{\pm}=\ran E_{N_2}(\pm(\tfrac{\|N\|}{n+1},\tfrac{\|N\|}{n}])$ for $n\ge1$; these reduce $T$ and, together with $\mathcal{H}_0$, give $\mathcal{H}_2=\mathcal{H}_0\oplus\bigoplus_{n,\pm}\mathcal{H}_n^\pm$. On $\mathcal{H}_0$, using Lemma~\ref{lem:nil} we have $T^2=0$. On $\mathcal{H}_n^\pm$, $N$ is invertible with $\pm N>0$; put $K=|N|^{1/2}$ (positive invertible, commuting with $T,T^*$) and $T^2=\pm K^2$. $T$ is invertible. The orthogonal direct sum $J$ of the countably many imaginary operators so obtained is imaginary and commutes with $T$.
	\end{proof}
	
	Quaternionic Fuglede-Putnam theorem was studied by Benabdi and Barraa \cite{BB20}. Now we give an alternative proof.
	{\begin{theorem}\label{thm:FP}
			Let $\mathcal{H}_1,\mathcal{H}_2$ be right quaternionic Hilbert spaces, $N_1\in\mathcal{B}(\mathcal{H}_1)$ and $N_2\in\mathcal{B}(\mathcal{H}_2)$ normal, and $S\in\mathcal{B}(\mathcal{H}_2,\mathcal{H}_1)$. If $N_1S=SN_2$, then $N_1^*S=SN_2^*$.
	\end{theorem}}
	
	\begin{proof}
		Fix $m\in\mathbb{S}$. Then there are unitaries $U_k:\mathcal{H}_k\to L^2(\Omega_k;\mathbb{H};\mu_k)$ with $N_k=U_k^*M_{\varphi_k}U_k$, where $\varphi_k: \Omega_k\to \mathbb{C}_m$. Let $f=a+jb$ with $a,b\in L^2(\Omega,\mathbb{C}_m,\mu)$. Since $\varphi j=j\bar\varphi$,
		\[
		\varphi(a+jb)=\varphi a+j\bar\varphi b.
		\]
		So, $\phi$ can be expressed as a $\mathbb{C}_m$-linear operator on $L^2(\Omega,\mathbb{C}_m,\mu)^2$ by $M_\varphi=M_\varphi\oplus M_{\bar\varphi}$. Since $M_\varphi^*=M_{\bar\varphi}$, we have $M_{\bar\varphi}\oplus M_\varphi=(M_\varphi\oplus M_{\bar\varphi})^*$. If $\tilde S=U_1SU_2^*$, then $\tilde{S}$ is a $\mathbb{C}_m$-linear operator and intertwines the two complex normal multiplication operators $\Phi_k=M_{\varphi_k}\oplus M_{\bar\varphi_k}$, $k=1,2$. For, $N_1S=SN_2$ implies $M_{\varphi_1}\tilde{S}=\tilde{S}M_{\varphi_2}$. Now,
		\begin{align*}
			\Phi_1 \tilde{S}&= (M_{\varphi_1}\oplus M_{\bar\varphi_1})\tilde{S}\\
			&= M_{\varphi_1}\tilde{S}\oplus M_{\bar\varphi_1}\tilde{S}\\
			&= \tilde{S}M_{\varphi_2}\oplus \tilde{S}M_{\bar\varphi_2}\\
			&= \tilde{S}(M_{\varphi_2}\oplus M_{\bar\varphi_2})\\
			&=\tilde{S} \Phi_2 .
		\end{align*}
		
		The classical Fuglede-Putnam theorem for $M_{\varphi_k}$ gives $\Phi_1^*\tilde S=\tilde S\Phi_2^*$, which in turn gives $N_1^*S=SN_2^*$.
	\end{proof}

	{\begin{corollary}\label{cor:fuglede}
			Let $N\in\mathcal{B}(\mathcal{H})$ be normal.
			\begin{enumerate}
				\item[(a)] If $TN=NT$ then $TN^*=N^*T$; consequently $T^*$ also commutes with $N$ and $N^*$.
				\item[(b)] If $TN=NT$ then $T$ commutes with $N_1=\frac12(N+N^*)$, $N_2=\frac12(N-N^*)$, $|N_2|$, and with the spectral projections $E_N(\Omega)$ (all axially symmetric $\Omega$). Moreover $\ker N_2$ and $(\ker N_2)^\perp$ both reduce $T$ and $N$.
				\item[(c)] If $N_1S=SN_2$, then $E_{N_1}(\Omega)S=SE_{N_2}(\Omega)$ for every axially symmetric Borel $\Omega$.
			\end{enumerate}
	\end{corollary}}
	
	{\begin{lemma}\label{thm:main}
			Let $\mathcal{H}$ be a right quaternionic Hilbert space and $T\in\mathcal{B}(\mathcal{H})$. The following are equivalent.
			\begin{enumerate}
				\item[(a)] $T^2$ is normal.
				\item[(b)] There are right quaternionic Hilbert spaces $\mathcal{H}_A,\mathcal{H}_B$, normal operators $A\in\mathcal{B}(\mathcal{H}_A)$, $B\in\mathcal{B}(\mathcal{H}_B)$, a positive one-to-one $C\in\mathcal{B}(\mathcal{H}_B)$ with $BC=CB$, and a unitary $V:\mathcal{H}\to\mathcal{H}_A\oplus\mathcal{H}_B\oplus\mathcal{H}_B$ such that
				\[
				VTV^*=A\oplus\begin{pmatrix}B&C\\0&-B\end{pmatrix}.
				\]
			\end{enumerate}
	\end{lemma}}
	
	\begin{proof}
		\emph{(b)$\Rightarrow$(a).} We have $T^2=A^2\oplus \begin{pmatrix}B^2&BC-CB\\0&B^2\end{pmatrix} = A^2\oplus(B^2\oplus B^2)$, which is normal.
		
		\emph{(a)$\Rightarrow$(b).}
		
		By Proposition~\ref{prop:J}, 
		there is an imaginary $J$ with $TJ=JT$. Fix $i\in\mathbb{S}$ and let $\mathcal{H}_+=\mathcal{H}_+^J$, a complex Hilbert space.
		Hence, $T_+:=T|_{\mathcal{H}_+}\in\mathcal B(\mathcal{H}_+)$, and $(T_+)^2=(T^2)_+$ is normal because $T^2$ commutes with $J$  and the restriction map is a $*$-homomorphism. Then there are complex Hilbert spaces $\mathcal X_A,\mathcal X_B$, complex normal $A',B'$ with $\sigma(B')$ in the closed upper half-plane and non-negative Hermitian part, a one-to-one $C'\ge0$ commuting with $B'$ ,  $B'^*$, and a unitary $V_+:\mathcal{H}_+\to\mathcal X_A\oplus\mathcal X_B\oplus\mathcal X_B$ with $V_+T_+V_+^*=A'\oplus\begin{pmatrix}B'&C'\\0&-B'\end{pmatrix}$ \cite{RR71}.
		Let $\mathcal{H}_A=(\mathcal X_A)_\mathbb{H}$, $\mathcal{H}_B=(\mathcal X_B)_\mathbb{H}$. Then $(\mathcal X_A\oplus\mathcal X_B\oplus\mathcal X_B)_\mathbb{H}=\mathcal{H}_A\oplus\mathcal{H}_B\oplus\mathcal{H}_B$. The extension $V=\tl{V_+}$ is a unitary $\mathcal{H}\to\mathcal{H}_A\oplus\mathcal{H}_B\oplus\mathcal{H}_B$, and since $S\mapsto\tl S$ preserves sums, products and adjoints,
		\[
		VTV^*=\tl{A'}\oplus\begin{pmatrix}\tl{B'}&\tl{C'}\\0&-\tl{B'}\end{pmatrix},
		\]
		where $A:=\tl{A'}$ and $B:=\tl{B'}$ are normal, $C:=\tl{C'}=(\tl{\,C'^{1/2}})^*(\tl{\,C'^{1/2}})\ge0$ is one-to-one ($\tl{C'}(x_1+x_2j)=C'x_1+(C'x_2)j$), and $BC=CB$ since $B'C'=C'B'$.

	\end{proof}
	Now we characterize the model theorem for quaternionic sub-2-normal operators. We show that quaternionic sub-2-normal operators are unitarily equivalent to block multiplication operators on vector valued Hardy spaces in a similar manner in \cite{curtoprasad}.	
	\begin{theorem} \label{thmcyclic} 
		Let $T\in B(\mathcal{H})$ be a cyclic $2$-normal operator. \ Then there exists a finite measure space $(\Omega,\mu)$ and a  $2$-normal symbol $\Psi=\begin{bmatrix}
		\psi_1 & \psi_2\\ \psi_3 &\psi_4
		\end{bmatrix}$; $\psi_i \in L^\infty (\Omega, \mathbb{H}, \mu)$ such that $T$ is unitarily equivalent to $M_{\Psi} \in B(L^{2}(\Omega,\mathbb{H},\mu)\oplus L^{2}(\Omega,\mathbb{H},\mu))$. 
	\end{theorem}
	
	\begin{proof}
		Since $T$ is a cyclic quaternionic  $2$--normal operator, by Lemma ~ \ref{thm:main}, $T$ can be expressed as,
		$$
		T=\begin{bmatrix}
			B& C\\
			0& -B\\
		\end{bmatrix}.
		$$ where $B$ is a cyclic quaternionic normal operator, $C$ is a quaternionic positive operator and $B$ commutes with $C$.
		 Hence, there exists a finite measure space $(\Omega, \mathbb{H} ,\mu)$ and a function  $\varphi: \Omega\to \mathbb{C}_m$  such that $B$ is unitarily equivalent to the operator $ M_{\varphi}$ acting  on $L^{2}(\Omega, \mathbb{H},\mu)$. Since $B$ is cyclic, we have $\varphi(z)=z$ \cite{ST for normal}. Then it follows that,
		$$T=
		(U \oplus U )(M_\Phi+M_\Psi)(U^*\oplus U^*),
		$$
		where
		$M_\Phi+M_\Psi=M_{\Phi +\Psi}=\begin{bmatrix}
			M_\varphi& 0\\
			0&  -M_\varphi
		\end{bmatrix} +\begin{bmatrix}
		0& UCU^*\\
		0&  0
		\end{bmatrix}$.
		Since $C$ is positive, hence normal and $BC=CB$, $C$ is unitarily equivalent to a multiplication operator $\psi : \Omega \to \mathbb{C}_m$ such that $\psi$ and $\varphi$ commutes. Therefore,
		$T=(U \oplus U )M_{\Phi+\Psi}(U^*\oplus U^*)$ where $M_{\Phi+\Psi}=\begin{bmatrix}
			z& z\\
			0&  -z
		\end{bmatrix}$  
		and  $ U \oplus U $ is a unitary operator from $L^2(\Omega, \mathbb{H}, \mu) \oplus L^2(\Omega, \mathbb{H}, \mu) $  to $\mathcal{H}$.
		
		The reverse part directly follows from the definition of $2$-normal operators.
	\end{proof}
	
	
	\begin{theorem} \label{t2}
		Every right linear quaternionic cyclic sub-2-normal operator is unitarily equivalent to a block multiplication operator $M_\eta$ acting on a suitable Hardy space $H^{2}(\Omega,\mathbb{H},\mu)\oplus H^{2}(\Omega,\mathbb{H},\mu)$, where
		$$
		\eta(z) := \begin{bmatrix}
			z& z \\
			0& -z
		\end{bmatrix} \quad (z \in \mathbb{H}).
		$$
		
	\end{theorem}
	
	\begin{proof}
		Let $N$ be the cyclic $2$-normal extension of the cyclic sub-2-normal operator $S$. From Theorem~\ref{thmcyclic}, it follows that $N$ is unitarily equivalent to a block multiplication operator $M_{\eta}$ on $L^2(\Omega, \mathbb{H},\mu) \oplus L^2(\Omega, \mathbb{H},\mu)$.
		
		Now we have, 
		\[\|p(S)x\|^2=\int\limits_\Omega |p(\eta)(y)|^2d\mu)\] 
		where $y=(U\oplus U)^{*}x \in L^{2}(\Omega,\mathbb{H},\mu) \oplus L^{2}(\Omega,\mathbb{H},\mu)$. We define   $W(p(\eta))(y):= p(S)x$, then $ W$ is unitary. Let $q(z):=zp(z)$ where $z \in \mathbb{H}$, and observe that $Sp(S)x = q(S)x$. 
		
		Now we have 
		\[	W^{-1}S W p(\eta)(z) (y)= q(\eta(z))y=\begin{bmatrix}
			z& z\\
			0& - z
		\end{bmatrix}  p\left( \begin{bmatrix}
			z& z\\
			0&  -z\
		\end{bmatrix} \right) (y)=	M_{  \eta(z)} p(\eta(z))(y).\]
		Therefore, $T$ is unitarily equivalent to  $M_{\eta(z)}$.
	\end{proof} 
		\section{Extensions of quaternionic subnormal operators and quaternionic Quantum $q$-Oscillator in Economics}
	The $q$-deformation of the quantum harmonic oscillator called the q-oscillator has a vital role in quantum Physics and quantum Economics. The $q$-oscillator algebra which is generated by three objects spanned by $A$, $B$ and unity of algebra, $1$ satisfying the commutation relation,
	\[AB-qBA=1,\]
	 see \cite{q-osci} for more details.
	Then using the adjoints in Hilbert spaces, the operator based commutation relation,
	\begin{equation}\label{qqo}
		S^*S-qSS^*=I.
		\tag{$\mathcal{O({\rm q},S,H)}$}
	\end{equation}
	The complex versions of the more interpretations of the above relation can be seen in \cite{q-osci}. The corresponding quaternionic versions can be easily verified. 
	
	$(\mathcal{O({\rm q},S,H)})$ implies that
	\begin{gather*}
		S \text{ is closed, } \mathcal{D} \text{ is dense in } \mathcal{H} \text{ and} \\
		\mathcal{D} \subset \mathcal{D}(S^*\overline{S}) \cap \mathcal{D}(\overline{S}S^*), \ \ S^*Sf - qSS^*f = f, \ \ f \in \mathcal{D}. \tag{$\mathcal{O({\rm q},S,D)}$}
	\end{gather*}
	The other interpretation of  $(\mathcal{O({\rm q},S,H)})$ is,
	\begin{equation}
		\langle Sf, Sg \rangle - q\langle S^*f, S^*g \rangle = \langle f, g \rangle, \ \ f,g \in \mathcal{D}(S) \cap \mathcal{D}(S^*). \tag{$\mathcal{O({\rm q},S,{\rm w})}$}
	\end{equation}
	Since $\mathcal{O({\rm q},S,{\rm w})}$ is equivalent to
	\begin{equation*}
		\|Sf\|^2 - q\|S^*f\|^2 = \|f\|^2, \ \ f \in \mathcal{D}(S) \cap \mathcal{D}(S^*),
	\end{equation*}
	Since $S$ is closed and holds $\mathcal{O({\rm q},S,{\rm w})}$, the above relation is equivalent to
	\begin{equation*}
		\langle \overline{S}f, \overline{S}g \rangle - q \langle S^*f, S^*g \rangle = \langle f, g \rangle, \quad f \in \mathcal{D}(\overline{S}) \cap \mathcal{D}(S^*).
	\end{equation*}
	Now we state some variations of $(\mathcal{O({\rm q},S,H)})$.
	\begin{enumerate}
		\item 
		If $(\mathcal{O({\rm q},S,D)})$ holds with $\mathcal{D}$ being a core of $S$ then $(\mathcal{O({\rm q},S,{\rm w})})$ holds and $\mathcal{D}(\overline{S}) \subset \mathcal{D}(S^*)$. 
		\item  If $(\mathcal{O({\rm q},S,D)})$ holds with $\mathcal{D}$ being a core of $S^* $ then $(\mathcal{O({\rm q},S,{\rm w})})$ holds and $\mathcal{D}(S^*) \subset \mathcal{D}(\overline{S})$. 
		
		\item 
		If $(\mathcal{O({\rm q},S,{\rm w})})$ holds, then $(\mathcal{O({\rm q},S,D)})$ holds with $\mathcal{D} = \mathcal{D}(S^*\overline{S}) \cap \mathcal{D}(\overline{S}S^*)$. 
		\item 
		If $(\mathcal{O({\rm q},S,{\rm w})})$ holds and $\mathcal{D}(\overline{S}) \cap \mathcal{D}(S^*)$ is a core of $S$ and $S^*$ then $\mathcal{D}(S^*\overline{S}) = \mathcal{D}(\overline{S}S^*)$.
		\item 
		If $(\mathcal{O({\rm q},S,{\rm w})})$ holds and $\mathcal{D}(\overline{S}) = \mathcal{D}(S^*)$ then $\overline{S}$ satisfies $(\mathcal{O({\rm q},S,D)})$ on $\mathcal{D} = \mathcal{D}(S^*\overline{S}) = \mathcal{D}(\overline{S}S^*)$.
	\end{enumerate}
	The above relations can be obtained using the proof techniques in \cite{q-osci}.
	
	For an integer $z$ and a real number $q$, we define $[z]_q= (1 - q^z)(1 - q)^{-1}$ if $q \neq 1$ and $[z]_1 = z$. If $z$ is a non-negative integer, then $[z]_q = 1 + q +q^2+ \cdots + q^{z-1}$ and this is usually called $q$-number \cite{q-osci}.	Suppose $q > 0$ and $S$ is a right linear weighted shift with respect to an orthonormal basis $(e_k)_{k=0}^\infty$ of the right quaternionic Hilbert space $\mathcal{H}$ with the real weights $(\sqrt{[k + 1]_q})_{k=0}^\infty$. If we define,
	\[S_0 = S, \quad S_n = q^{n/2}S, \quad D_n = \sqrt{[n]_q} \, \text{diag}(q^{k/2})_{k=0}^\infty, \quad n = 1, 2 \dots, \]
	the matrix operator given by
	\begin{equation}\label{eq:matrix}
		N=	\begin{pmatrix}
			S_0    & D_1    & 0      & 0      & \dots  \\
			0      & S_1    & D_2    & 0      & \dots \\
			0      & 0      & S_2    & D_3    & \dots \\
			\vdots & \vdots & \vdots & \vdots & \ddots 
		\end{pmatrix}
	\end{equation}
	defines a right linear operator in the quaternionic direct sum space $\bigoplus\limits_{n=0}^\infty \mathcal{H}_n$ with $\mathcal{H}_n = \mathcal{H}$, whose domain is composed of all those algebraic direct sums $\bigoplus\limits_{n=0}^\infty f_n$ for which $f_n = 0$ except for a finite number of indices $n$. Here using definition of adjoint,
	\[		\mathcal{D}(N^*) = \bigoplus\limits_{n=0}^\infty \mathcal{D}(S_n^*).\]
	Now, we have the norm equality $\|Nf\| = \|N^*f\|$ satisfies on $\mathcal{D}(N)$, so that $N$ is normal. Consequently, $
	S$  is subnormal and $\overline{N}$  is its tight and $\ast$-tight normal extension.
	\begin{theorem}
	Let $\mathcal{H}$ be a seperable right quaternionic Hilbert space and $S$ be a densely defined closed right linear operator on $\mathcal{H}$ with right orthonormal basis $\{e_n\}_{n=0}^\infty \subset \mathcal{D}(\overline{S})$ and  For a real parameter $q > 0$, 
		\[
		\overline{S}e_n = \sqrt{[n+1]_q} \, e_{n+1}, \quad n=0,1,2,\dots
		\]
		if and only if  $S$ is irreducible, satisfies $(\mathcal{O({\rm q},S,H)})$ with a domain $\mathcal{D}$ that is a core for $S$ and invariant under both $S$ and $S^*$, and $S$ is subnormal possessing a tight and $\ast$-tight normal extension.
	\end{theorem}
	
	\begin{proof}
		Assume $\overline{S}$ be defined as above. Let $\mathcal{D}$ be the right linear span of $\{e_n\}_{n=0}^\infty$. Since the weights $\sqrt{[n+1]_q}$ are real, $S(\mathcal{D}) \subset \mathcal{D}$ and $S^*(\mathcal{D}) \subset \mathcal{D}$ so that $\mathcal{D}$ is a reducing subspace. Using the entrywise matrix normal extension established for quaternionic direct sums, $S$ is subnormal and admits a tight and $\ast$-tight normal extension $\overline{N}$. Since $S$ is a unilateral right linear weighted shift with non-zero real weights, any closed right linear subspace reducing $S$ is either $\{0\}$ or $\mathcal{H}$. So $S$ is irreducible.
		
		Conversely if $S$ has a tight and $\ast$-tight normal extension, $\mathcal{D}(\overline{S}) = \mathcal{D}(S^*)$. Since $\mathcal{D}$ is given to be a core for $S$, the relation $\|S^*f\|^2 = \frac{1}{q}(\|Sf\|^2 - \|f\|_2^2)$ ensures that any sequence in $\mathcal{D}$ converges simultaneously in the graph norm of $S^*$. Hence, $\mathcal{D}$ is a core for both $S$ and $S^*$. Now taking the limits of the inner product relation $\langle Sf, Sg \rangle - q\langle S^*f, S^*g \rangle = \langle f, g \rangle$ for $f,g \in \mathcal{D}$ gives $(\mathcal{O({\rm q},S,\rm{w})})$ on $\mathcal{D}$. Hence $\mathcal{D}(\overline{S}) = \mathcal{D}(S^*)$ follows from the closure of the core intersections. Hence, $\mathcal{D}(\overline{S}) \cap \mathcal{D}(S^*)$ is dense in $\mathcal{H}$. If the right linear kernel satisfied $\mathcal{N}(S^*) = \{0\}$, the operator $S^*$ would be injective. Combined with irreducibility and the $q$-commutation norm identity, contradicts the assumption of subnormality. Thus, $\mathcal{N}(S^*) \neq \{0\}$. By the inductive application of $S$, $S^n(\mathcal{N}(S^*)) \subset \mathcal{D}(S) \cap \mathcal{D}(S^*)$ for all $n = 0, 1, 2, \dots$
		
		Now we choose a non-zero vector $e_0 \in \mathcal{N}(S^*)$ such that $\|e_0\| = 1$. Applying the weak norm identity $\|Sf\|^2 = q\|S^*f\|^2 + \|f\|^2$, we have $S^*e_0 = 0$ and 
		\[
		\|S e_0\|^2 = q\|S^*e_0\|^2 + \|e_0\|^2 = 1 = [1]_q.
		\]
		Now, defining $e_n = \left([n]_q!\right)^{-1/2} S^n e_0$ constructs a right orthonormal sequence satisfying $\overline{S}e_n = \sqrt{[n+1]_q}e_{n+1}$. By the topological irreducibility of $S$ in $\mathcal{H}$, the closed right linear span of $\{e_n\}_{n=0}^\infty$ cannot be a proper subspace. Therefore, it forms a complete right orthonormal basis for $\mathcal{H}$.
	\end{proof}

	Since the classical quantum harmonic oscillator considers only one attribute, we are not able to get the full idea of the market or economy. While we are considering quaternionic quantum harmonic oscillator (QHO), we are able to consider different attributes and we can compare their correlation. Since in classical QHO, eigenvalues are single values, we are able to identify the amount of transaction or frequency at that state. But  in  quaternionic QHO, instead of a single value, we get a sphere symmetric to real axis from which we can analyse what value of energy is obtained in each direction and can identify which attributes give a more stable economy. 

In \cite{sqho}, authors studied quaternionic subnormal operators and proved that unbounded quaternionic subnormal operators belongs to the solution space of canonical commutation relation. Hence, we identify creation operators with unbounded subnormal operators. By using the model theory for unbounded quaternionic subnormal operators, we can  analyse the behavior of annihilation and creation operators, which yields the analysis of Hamiltonian operators. Also we give some spectral analysis of the quaternionic QHO using the S-spectra of quaternionic subnormal operators.

It is well known that if we let $q=1$, the quaternionic  quantum $q$-oscillator problem reduces to the quaternionic commutation relation. In \cite{qho}, authors proved that the unbounded quaternionic subnormal operators belongs to the solution space of quaternionic quantum commutation relation. We use the result to study the quaternionic Hamiltonian equation.
We  present  the Hamiltonian equation as,	\begin{equation}\label{hamil}
	\hat{H}x= \dfrac{\hbar\omega}{2} \left(\hat{a}^\dagger\hat{a}-\dfrac{1}{2}\right)x
\end{equation}
where  $\hat{a}$ and $\hat{a}^\dagger $ are annihilation and creation operators respectively. Since the unbounded quaternionic subnormal operator $S$  satisfy the canonical commutation relation as like annihilation and creation operators, we have,
\begin{equation}\label{hamil}
	\hat{H}x= \dfrac{\hbar\omega}{2} \left(S^*S-\dfrac{1}{2}I\right)x
\end{equation}
for every $x \in \mathcal{D}$ where $\mathcal{D}$ is invariant under both $S$ and $S^*$. Using model theorem for unbounded subnormal operators Theorem  \ref{M_z}, we can rewrite Equation(\ref{hamil}) as, 

\begin{equation}
	\hat{H}x= V^*\dfrac{\hbar\omega}{2} \left(M_{|\eta|^2}-\dfrac{1}{2}I\right)Vx
\end{equation}

for every $x \in \mathcal{D}$,
where $M_{|\eta|^2}$ is the multiplication operator on quaternionic Hardy space. That is $\hat{H}x = V^*M_\nu Vx$ where $\nu= \dfrac{\hbar \omega}{2}(|\eta|^2-1/2)$. It is clear that $M_\nu$ is a multiplication operator on quaternionic Hardy space. So every Hamiltonian operator is unitarily equiovalent to a multiplication operator on quaternionic Hardy space. Since $\sigma_S(\hat{H})=\sigma_S(M_{|\eta|^2})$, studying the  $S$-spectra of Hamiltonian operators can be easily done from the  $S$-spectra of multiplication operators on  quaternionic Hardy space.

Now we consider Dirac type oscillators. Define
	$
	M=
	\begin{pmatrix}
		f&g\\
		0&k
	\end{pmatrix}
	$ on the right quaternionic Hilbert space $\mathcal{H}=\mathcal{H}_1\oplus\mathcal{H}_2$ where $f:\mathcal{H}_1\to\mathcal{H}_1,
	g:\mathcal{H}_2\to\mathcal{H}_1,$ and $k:\mathcal{H}_2\to\mathcal{H}_2$. The adjoint operator is given by
	$
	M^*=
	\begin{pmatrix}
		f^*&0\\
		g^*&k^*
	\end{pmatrix}.
	$ where $f^*(x)=\overline{f(x)}$. Define the Hermitian dilation,
	\begin{equation*}
		\mathscr D_M=
		\begin{pmatrix}
			0&M^*\\
			M&0
		\end{pmatrix}.
		\label{eq:DM}
	\end{equation*}
	The Hermitian dilation operator defined on $\mathcal{H}\oplus\mathcal{H}$ is a self-adjoint operator. 	If we define $\Gamma=
	\begin{pmatrix}
		I_{\mathcal{H}}&0\\
		0&-I_{\mathcal{H}}
	\end{pmatrix}$ where $I_{\mathcal{H}}$ denotes the identity operator on $\mathcal{H}$, then $\Gamma\mathscr D_M+\mathscr D_M\Gamma=0$.
	Thus $\mathscr D_M$ has an off-diagonal grading and anticommutes with the grading operator, which is the characteristic algebraic structure of a supersymmetric Dirac operator. For more details, see \cite{dirac1, dirac2}.
	
	Now if we let $g= \hat{a}^\dagger$ (creation operator), that is;
	\begin{equation*}
		M=
		\begin{pmatrix}
			f&\hat{a}^\dagger\\
			0&k
		\end{pmatrix}.
		\label{eq:MI}
	\end{equation*}
	Here $f$ and $k$ may be scalar multiples of the identity, oscillator potentials, differential operators, or other closed right-linear operators, subject to the relevant
	domain conditions. Consequently,
	\begin{align*}
		M^* M
		=
		\begin{pmatrix}
			f^* f&f^* \hat{a}^\dagger\\
			\hat{a}f&\hat{a}\hat{a}^\dagger+k^* k
		\end{pmatrix}
		=
		\begin{pmatrix}
			f^* f&f^* \hat{a}^\dagger\\
			\hat{a}f&N+I_{\mathcal{H}}+k^* k
		\end{pmatrix},
		\label{eq:MIstarMI}
	\end{align*}
	and
	\begin{align*}
		MM^*
		=
		\begin{pmatrix}
			ff^*+\hat{a}^\dagger a&\hat{a}^\dagger k^*\\
			k\hat{a}&kk^*
		\end{pmatrix}
		=
		\begin{pmatrix}
			ff^*+N&\hat{a}^\dagger k^*\\
			k\hat{a}&kk^*
		\end{pmatrix}.
	\end{align*}
	Note that $\hat{a}^\dagger\hat{a}=N,$ the number operator and $\hat{a}\hat{a}^\dagger=I_\mathcal{H}+N$. Hence the harmonic oscillator is present in the products of the off-diagonal block, even though $f$ and $k$ are nonzero. 
	
	The Hermitian dilation
	\[
	\mathscr D_M=
	\begin{pmatrix}
		0&M^*\\
		M&0
	\end{pmatrix}
	\]
	provides a useful way to represent a generally non-self-adjoint triangular operator as a self-adjoint Dirac-type operator. The operator
	\[
	\mathscr D_M^2=
	\operatorname{diag}(M^* M,MM^*),
	\]
	revealing the factorization structure that also appears in supersymmetric quantum mechanics.

In this article, we characterize bounded quaternionic sub-$2$-normal operators which is unitarily equivalent to block multiplication operators on vector valued quaternionic Hardy space. If we 	characterize unbounded sub-2-normal operators in a similar way as in Theorem \ref{t2}, we can generate Dirac-type operators which involves in supersymmetric quaternionic quantum mechanics. In the setting of a right quaternionic Hilbert space, one can choose a complex slice \(\mathbb C_I\) and construct the canonical commutation relation
	\[
	[\hat{a},\hat{a}^\dagger]=1,
	\]
	from which the oscillator number operator and Hamiltonian can be defined. Since we have already shown that quaternionic unbounded subnormal operators belongs to the solution space of canonical commutation relation, unbounded subnormal operators can generate Dirac type operators. 
	\bibliographystyle{abbrv}

\begin{thebibliography}{10}
		\bibitem{adler}
		S. L. Adler,  Quaternionic Quantum Mechanics and Quantum Fields. Oxford University Press. (1995).
		
		\bibitem{ACK16} 
		D. Alpay, F. Colombo, D. P. Kimsey, The spectral theorem for quaternionic unbounded normal operators based on the S-spectrum, J.\ Math.\ Phys.\ 57 (2016), 023503.
		
		
		\bibitem{BB20}
		E.~H.~Benabdi, M.~Barraa, {The spectral theorem for quaternionic normal operators}, arXiv:2006.05253.
		
	\bibitem{dirac3}
	J. Bentez, R.M. y Romero, H.N. Núez-Yépez, A. L. Salas-Brito,  Solution and hidden supersymmetry of a Dirac oscillator. Physical review letters, 64(14), p.1643 (1990).
		
		\bibitem{von N}
		G. Birkhoff,  J. von Neumann,  The Logic of Quantum Mechanics. Annals of Mathematics, 37(4), 823–843 (1936). 
		
		\bibitem{bram1955}
		J. B. Bram, {Subnormal operators}, 
		Duke Math. J. \textbf{22}, no. 1, 75--94 (1955).
		
	\bibitem{CGK18}
	F.~Colombo, J.~Gantner, D.~P.~Kimsey, \emph{Spectral Theory on the S-Spectrum for Quaternionic Operators}, Oper.\ Theory Adv.\ Appl.\ 270, Birkh\"auser, 2018.
		
		\bibitem{colombo2016}
		F.~Colombo, J.~Gantner, and I.~Sabadini,
		\emph{Spectral Theory on the $S$-Spectrum for Quaternionic Operators},
		Operator Theory: Advances and Applications, Birkh\"{a}user, 2016.
		
		\bibitem{colombo2009}
		F.~Colombo and I.~Sabadini,
		``The $S$-spectrum and the functional calculus for quaternionic operators,''
		\emph{Journal of Functional Analysis}, Vol. 256, No. 11, pp. 3770--3788, (2009).
		
		\bibitem{sabadini}
		F. Colombo, I. Sabadini, and  D. C. Struppa,   Noncommutative Functional Calculus: Theory and Applications of S-Spectrum. Springer Science \& Business Media (2011).
		
		
		
		
		
		\bibitem{subnormal}
		J. B. Conway, The theory of subnormal operators. No. 36. American Mathematical Soc., (1991).	
		
		\bibitem{curto}
		R.E. Curto, T. Prasad. Classes of operators related to subnormal operators,  Revista de la Real Academia de Ciencias Exactas, Físicas y Naturales. Serie A. Matemáticas 120.2 (2026).
		
		\bibitem{curtoprasad}
		R.E. Curto, T. Prasad. Model Theory for Operators Related to Square Roots of Normal Operators. 10.48550/arXiv.2607.21450 (2026). 
		
		
		\bibitem{GMP13}
		R.~Ghiloni, V.~Moretti, A.~Perotti, \emph{Continuous slice functional calculus in quaternionic Hilbert spaces}, Rev.\ Math.\ Phys.\ 25 (2013), 1350006.
		
		\bibitem{giardino2026}
		S. Giardino,
		The quantum harmonic oscillator and the real Hilbert space,
		\emph{Annals of Physics}, Vol. 465, (2026).
		
		\bibitem{giardino2021}
		S. Giardino,
		Quaternionic quantum harmonic oscillator,
		\emph{European Physical Journal Plus}, Vol. 136, No. 1, (2021).
		
		\bibitem{halmos1950}
		P. R. Halmos, {Normal dilations and extensions of operators}, Summa Brasil. Math. \textbf{2}, 125--134, (1950).
		
		\bibitem{sqho}
		K. Krishnan, T. Prasad and E. Shine Lal, Subnormal Operators  and Quantum Harmonic Oscillator  on Quaternionic Hilbert Spaces, \textit{Communicated}
		
		\bibitem{dirac1}
		M. Moshinsky,  A. Szczepaniak, The dirac oscillator. Journal of Physics A: Mathematical and General 22.17,  L817-L819 (1989).
		
		\bibitem{orrell2022}
		D.~Orrell,
		{Quantum Economics and Finance: An Applied Mathematics Introduction},
		3rd Edition, Panda Ohana Publishing, (2022).
		
		\bibitem{orrell}
		D. Orrell, The Quantum Stock Market: And the Road Not Taken in Finance. MIT Press, (2026).	
		
		\bibitem{RR71}
		H.~Radjavi, P.~Rosenthal, {On roots of normal operators}, J.\ Math.\ Anal.\ Appl.\ 34 (1971), 653--664.
		
		\bibitem{ST for normal}
		G. Ramesh, P. Santhosh Kumar, Spectral theorem for quaternionic normal operators: Multiplication form,
		Bulletin des Sciences Mathématiques, Volume 159,
		(2020).
		
		\bibitem{Ros58}
		M.~Rosenblum, {On a theorem of Fuglede and Putnam}, J.\ London Math.\ Soc.\ 33 (1958), 376--377.
		
		
		\bibitem{Szafrani}
		J. Stochel, F. H. Szafraniec, Circular invariance of the Weyl form of the canonical commutation relation. World Scientific (2000).
		
		\bibitem{lambert1}
		J. Stochel, F. H. Szafraniec, On normal extensions of unbounded operators, II, Acta Sci.	Math. (Szeged) 53, 153-177 (1989).
		
		\bibitem{q-osci}
		F. H. Szafraniec,  Operators of the $ q $-oscillator, Banach Center Publications,  Institute of Mathematics Polish Academy of Sciences, 78, 293-307, (2007).
		
		
		\bibitem{qho}
		F. H. Szafraniec, Subnormality in the Quantum Harmonic Oscillator,	Commun. Math. Phys. 210, 323– 334 (2000).
		
		\bibitem{dirac2}
		T. T. Tsutsui, A. A. Silva, A. S. de Castro, F. M. Andrade,  Optical perspective on the time-dependent Dirac oscillator. Frontiers in Physics, 14, (2026).
		
		\bibitem{Viswanath}
		K. Viswanath, {Spectral Theorem for Quaternionic Normal Operators}, Transactions of the American Mathematical Society, Vol. 162, 337--350, (1971).
		
		
		
		
		
		
		
		
		
		
		
		
		
		
		
		
	\end{thebibliography}

\end{document}